\documentclass[12pt]{amsart}
\usepackage{amssymb, latexsym, euscript}
\usepackage{amssymb}
\usepackage{latexsym}
\usepackage{amsmath}

      \def\sL{{\mathfrak L}}
\def\sM{{\mathfrak M}}   \def\sN{{\mathfrak N}}

   \def\dN{{\mathbb N}}

   \def\cH{{\mathcal H}}   
   \def\cK{{\mathcal K}}   \def\cL{{\mathcal L}}
      
      \def\cR{{\mathcal R}}
   \def\cT{{\mathcal T}}   \def\cU{{\mathcal U}}
\def\cV{{\mathcal V}}   \def\cW{{\mathcal W}}

   \def\bB{{\mathbf B}}

\def\dim{{\rm dim\,}}
\def\wt{\widetilde}

\def\ran{{\rm ran\,}}
\def\cran{{\rm \overline{ran}\,}}
\def\dom{{\rm dom\,}}

\def\clos{{\rm clos\,}}

\def\dim{{\rm dim\,}}

\let\xker=\ker \def\ker{{\xker\,}}

\def\uphar{{\upharpoonright\,}}

\def\wt{\widetilde}

\newtheorem{theorem}{Theorem}[section]
\newtheorem{lemma}[theorem]{Lemma}
\newtheorem{proposition}[theorem]{Proposition}
\newtheorem{corollary}[theorem]{Corollary}

\newtheorem{remark}[theorem]{Remark}

\def\wt{\widetilde}

\def\f{\varphi}

\def\uphar{{\upharpoonright\,}}

\numberwithin{equation}{section}
\begin{document}
\title[ mappings connected with parallel addition ]
{On the mappings connected with parallel addition of nonnegative operators}

\author[Yury Arlinski\u{\i}]{Yu.M. Arlinski\u{\i}}
\address{Department of Mathematical Analysis  \\
East Ukrainian National University  \\
Prospect Radyanskii, 59-a, Severodonetsk, 93400, Ukraine\\
and
Department of Mathematics, Dragomanov National Pedagogical University,
Kiev, Pirogova 9, 01601, Ukraine}
 \email{yury.arlinskii@gmail.com}

\begin{abstract}
We study a mapping $\tau_G$ of the cone ${\mathbf B}^+({\mathcal H})$ of bounded nonnegative self-adjoint operators in a complex Hilbert space ${\mathcal H}$ into itself. This mapping is defined as a strong limit of iterates of the mapping ${\mathbf B}^+({\mathcal H})\ni X\mapsto\mu_G(X)=X-X:G\in{\mathbf B}^+({\mathcal H})$, where $G\in{\mathbf B}^+({\mathcal H})$ and $X:G$ is the parallel sum.
We find explicit expressions for $\tau_G$ and establish its properties. In particular, it is shown that $\tau_G$ is sub-additive, homogeneous of  degree one, and its image coincides with  set of its fixed points which is the subset of ${\mathbf B}^+({\mathcal H})$, consisting of all $Y$ such that ${\rm ran\,} Y^{1/2}\cap{\rm ran\,} G^{1/2}=\{0\}$.   Relationships between $\tau_G$ and Lebesgue type decomposition of nonnegative self-adjoint operator are established and applications to the properties of unbounded self-adjoint operators with trivial intersections of their domains are given.
\end{abstract}
\maketitle
\tableofcontents
\section{Introduction}
We will use the following notations: $\dom A$, $\ran A$, and $\ker A$ are the domain, the range, and the kernel of a linear operator $A$, $\cran A$ and $\clos{\cL}$ denote the closure of $\ran A$ and of the set $\cL$, respectively. A linear operator $A$ in a Hilbert space $\cH$ is called
\begin{itemize}
\item bounded from bellow if $(A f,f)\ge m||f||^2 $ for all $f\in\dom A$ and some real number $m$,
 \item positive definite if $m>0$,
  \item nonnegative if $(A f,f)\ge 0 $ for all $f\in\dom A.$
\end{itemize}
  The cone of all bounded self-adjoint non-negative operators in a complex Hilbert space $\cH$ we denote by $\bB^+(\cH)$ and let $\bB^{+}_0(\cH)$ be the subset of operators from $\bB^+(\cH)$ with
trivial kernels. If $A,B\in \bB^+(\cH)$ and $C=ABA$, then by Douglas theorem \cite{Doug} one has $\ran C^{1/2}=A\ran B^{1/2}$.
If $\cK$ is a subspace (closed linear manifold) in $\cH$, then $P_\cK$ is the orthogonal projection in $\cH$ onto $\cK$, and $\cK^\perp\stackrel{def}{=}\cH\ominus\cK$.

Let $X,G\in\bB^+(\cH)$.
 The
\textit{parallel sum} $X:G$ is defined by the quadratic form:
\[
\left((X:G)h,h\right)\stackrel{def}{=}\inf_{f,g \in \cH}\left\{\,\left(Xf,f\right)+\left(Gg,g\right):\,
   h=f+g \,\right\} \ ,
\]
see \cite{AD}, \cite{FW}, \cite{K-A}. One can establish for $X:G$  the following equivalent
definition  \cite{AT}, \cite{PSh}
\[
X:G=s-\lim\limits_{\varepsilon\downarrow 0}\,
X\left(X+G+\varepsilon I\right)^{-1}G.
\]
Then for positive definite bounded self-adjoint operators $X$ and $G$ we obtain
\[
X:G=(X^{-1}+G^{-1})^{-1} \ .
\]
As is known \cite{PSh}, $X:G$ can be calculated as follows
\[
X:G=X-\left((X+G)^{-1/2}X\right)^*\left((X+G)^{-1/2}X\right).
\]
Here for $A\in\bB^+(\cH)$ by $A^{-1}$ we denote the Moore--Penrose
pseudo-inverse.
The operator $X:G$ belongs to $\bB^+(\cH)$ and, as it is established in \cite{AT}, the equality
\begin{equation}
\label{ukfdyj}
\ran
(X:G)^{1/2}=\ran X^{1/2}\cap\ran G^{1/2}
\end{equation}
holds true.
If $T$ is bounded operator in $\cH$, then in general
$$T^*(A:B)T\le (T^*AT):(T^*BT)$$
for $A,B\in\bB^+(\cH)$, but, see  \cite{Ar2},
\begin{multline}
\label{trans}\ker T^*\cap\ran (A+B)^{1/2}=\{0\} \\
\Longrightarrow  T^*(A:B)T= (T^*AT):(T^*BT).
\end{multline}
Besides, if $A'\le A''$, $B'\le B''$,  then $A':B'\le A'':B''$ and, moreover \cite{PSh},
\begin{equation}
\label{monotcon}
A_n\downarrow A\quad\mbox{and}\quad B_n\downarrow B\quad\mbox{strongly}\Rightarrow A_n:B_n\downarrow A:B\quad\mbox{strongly}.
\end{equation}
Let $X,G\in\bB^+(\cH)$. Since $X\le X+G$ and $G\le X+G$, one gets
\begin{multline}
\label{fg2}
 X=(X+G)^{1/2}M(X+G)^{1/2},\\
   G=(X+G)^{1/2}(I-M)(X+G)^{1/2}
\end{multline}
for some non-negative contraction $M$ on $\cH$ with $\ran M\subset\cran(X+G)$.
\begin{lemma} {\rm \cite{Ar2}}
\label{yu1} Suppose $X, G\in \bB^+(\cH)$ and let $M$ be as in
\eqref{fg2}. Then
\[
 X:G=(X+G)^{1/2}(M-M^2)(X+G)^{1/2}.
\]
\end{lemma}
Since
$$
\ran M^{1/2}\cap\ran (I-M)^{1/2}=\ran (M-M^2)^{1/2},
$$
the next proposition is an immediate consequence of Lemma \ref{yu1}, cf. \cite{FW}, \cite{PSh}.
\begin{proposition}
\label{root} 1) $\ran (X:G)^{1/2}=\ran X^{1/2}\cap\ran G^{1/2}$.

2) The following statements are equivalent:
\begin{enumerate}
\def\labelenumi{\rm (\roman{enumi})}
\item
 $X:G=0$;
 \item
$M^2=M$, i.e., the operator $M$ in \eqref{fg2} is an orthogonal projection in
$\cran(X+G)$;
\item $\ran X^{1/2}\cap\ran G^{1/2}=\{0\}$.
\end{enumerate}
\end{proposition}

 Fix $G\in\bB^+(\cH)$ and define a mapping
\begin{equation}
\label{mapmu}
\bB^+(\cH)\ni X\mapsto\mu_G(X)\stackrel{def}{=}X-X:G\in \bB^+(\cH).
\end{equation}
Then
\begin{enumerate}
\item $0\le\mu_G(X)\le X$,
\item $\mu_G(X)=X\iff X:G=0\iff\ran X^{1/2}\cap \ran G^{1/2}=\{0\}$.
\end{enumerate}
Therefore, if $G$ is positive definite, then the set of fixed points of $\mu_G$ consists of a unique element, the trivial operator.
Denote by $\mu^{[n]}_G$ the $n$th iteration of the mapping $\mu_G$, i.e., for $X\in\bB^+(\cH)$
\begin{multline*}
\mu^{[2]}_G(X)=\mu_G(\mu_G(X)),\;\mu^{[3]}_G(X)=\mu_G(\mu^{[2]}_G(X)),\cdots,\\
\mu^{[n]}_G(X)=\mu_G(\mu^{[n-1]}_G(X)).
\end{multline*}
Since
\[
X\ge\mu_G(X)\ge \mu^{[2]}_G(X)\ge\cdots\ge\mu^{[n]}_G(X)\ge \cdots,
\]
the strong limit of $\{\mu^{[n]}_G(X)\}_{n=0}^\infty$ exists for an arbitrary $X\in\bB^+(\cH)$ and is an operator from $\bB^+(\cH)$.
In this paper we study the mapping
\[
\bB^+(\cH)\ni X\mapsto\tau_G(X)\stackrel{def}{=}s-\lim\limits_{n\to\infty}\mu^{[n]}_G(X)\in\bB^+(\cH).
\]
We show that the range and the set of fixed points of $\tau_G$ coincides with the cone
\begin{multline*}
\bB^+_G(\cH)=\left\{Y\in\bB^+(\cH): \ran Y^{1/2}\cap\ran G^{1/2}=\{0\}\right\}\\
=\left\{Y\in \bB^+(\cH), \;Y:G=0\right\}.
\end{multline*}
We find explicit expressions for $\tau_G$ and establish  its properties. In particular, we show that $\tau_G$ is  homogenous and sub-additive, i.e., $\tau_G(\lambda X)=\lambda\tau_G(X)$ and
$\tau_G(X+Y)\le \tau_G(X)+\tau_G(Y)$ for an arbitrary operators $X, Y\in\bB^+(\cH)$ and an arbitrary positive number $\lambda$. It turns out that
$$\tau_G(X)=\tau_{\wt G}(X)=\tau_G(\wt G+X)$$ for all $X\in\bB^+(\cH)$, where $\wt G\in\bB^+(\cH)$ is an arbitrary operator such that $\ran \wt G^{1/2}=\ran G^{1/2}$. We prove the equality  $\tau_G(X)=X-[G]X,$
where the mapping
\[
\bB^+(\cH)\ni X\mapsto[G]X\stackrel{def}{=}s-\lim\limits_{n\to\infty}(nG:X)\in\bB^+(\cH)
\]
has been defined and studied by T.~Ando \cite{Ando_1976} and then in \cite{Pek_1978}, \cite{Kosaki_1984}, and \cite{E-L}. In the last Section \ref{applll} we apply the mappings $\{\mu^{[n]}_G\}$ and $\tau_G$  to the problem of the existence of a self-adjoint operator
whose domain has trivial intersection with the domain of given unbounded self-adjoint operator \cite{Neumann}, \cite{Dix}, \cite{FW}, \cite{ES}. Given an unbounded self-adjoint operator $A$, in Theorem \ref{ytcgjl} we suggest several assertions equivalent to the existence of a unitary operator $U$ possessing the property $U\dom A\cap\dom A=\{0\}$. J.~von Neumann \cite[Satz 18]{Neumann} established that such $U$ always exists for an arbitrary unbounded self-adjoint $A$ acting in a separable Hilbert space. In a nonseparable Hilbert space  always exists an unbounded self-adjoint operator $A$ such that for any unitary $U$ the relation $U\dom A\cap\dom\ne\{0\}$ holds, see \cite{ES}.
\section{The mapping $\mu_G$ and strong limits of its orbits}

\begin{lemma}
\label{vspm}
Let $F_0\in\bB^+(\cH)$. Define the orbit
\[
F_1=\mu_G(F_0),\; F_2=\mu_G(F_1),\ldots,F_{n+1}=\mu_G(F_n),\ldots.
\]
Then the sequence $\{F_n\}$ is non-increasing:
$$F_0\ge F_1\ge\cdots \ge F_n\ge F_{n+1}\ge\cdots,$$
and the strong limit
\[
F\stackrel{def}{=}s-\lim\limits_{n\to\infty} F_n
\]
is a fixed point of $\mu_G$, i.e.,
satisfies the condition
\[
F:G=0.
\]
\end{lemma}
\begin{proof} Since $\mu_G(X)\le X$ for all $X\in\bB^+(\cH)$, the sequence $\{F_n\}$ is non-increasing.
Therefore, there exists  a strong limit $F=s-\lim\limits_{n\to\infty} F_n.$
On the other hand, because the sequence $\{F_n\}$ in non-increasing,
the sequence $\{F_n:G\}$ is non-increasing as well and property \eqref{monotcon} of parallel addition leads to
\[
s-\lim\limits_{n\to\infty}(F_n:G)=F:G.
\]
Besides, the equalities
\[
F_n:G=F_n-F_{n+1}, \;n=0,1,\ldots
\]
yield $F:G=0.$
 Thus, $F=\mu_G(F)$, i.e., $F$ is a fixed point of the
mapping $\mu_G$.
\end{proof}

For $G,F_0\in\bB^+(\cH)$ define subspaces
\begin{equation}
\label{prosm}\begin{array}{l}
\Omega\stackrel{def}{=}{\rm{clos}}\left\{f\in\cH:(G+F_0)^{1/2}f\in\ran G^{1/2}\right\}, \\
 \sM\stackrel{def}{=}\cH\ominus\Omega.
 \end{array}
\end{equation}
Note that if a linear operator $\cV$ is defined by
\begin{equation}
\label{contrv}
\left\{\begin{array}{l}x=(G+F_0)^{1/2}f+g\\
\cV x=G^{1/2}f,\; f\in\cH,\; g\in\ker(G+F_0)
\end{array}
\right.,
\end{equation}
then $\dom \cV=\ran(G+F_0)^{1/2}\oplus\ker(G+F_0)$ is a dense in $\cH$ linear manifold and $\cV$ is a contraction. Let $\overline{\cV}$ be the continuation of $\cV$ on $\cH$. Clearly $\overline{\cV}=\cV^{**}$.
If we denote by $(G+F_0)^{-1/2}$ the Moore-Penrose pseudo-inverse to $(G+F_0)^{1/2}$, then from \eqref{contrv} one can get that
\begin{equation}
\label{opercv}
\begin{array}{l}
\cV(G+F_0)^{1/2}=G^{1/2}=(G+F_0)^{1/2}\cV^*,\\
\cV^*=(G+F_0)^{-1/2}G^{1/2},\;\ran \cV^*\subseteq\cran(G+F_0),\\
\cV g=G^{1/2}(G+F_0)^{-1/2}g,\; g\in\ran(G+F_0)^{1/2}.
\end{array}
\end{equation}
Moreover,
\begin{equation}
\label{11}
\Omega=\cran \cV^*\oplus\ker(G+F_0), \; \sM=\ker \left(\overline{\cV}\uphar\cran(G+F_0)\right).
\end{equation}
Besides we define the following contractive linear operator
\begin{equation}
\label{contrw}
\left\{\begin{array}{l}x=(G+F_0)^{1/2}f+g\\
\cW x=F_0^{1/2}f,\; f\in\cH,\; g\in\ker(G+F_0).
\end{array}
\right.
\end{equation}
The operator $\cW$ is defined on $\dom \cW=\ran(G+F_0)^{1/2}\oplus\ker(G+F_0)$ and
\begin{equation}
\label{opwa}
\begin{array}{l}
\cW(G+F)^{1/2}=F^{1/2}_0=(G+F_0)^{1/2}\cW^*,\\
\cW^*=(G+F_0)^{-1/2}F^{1/2}_0,\;\ran \cW^*\subseteq\cran (G+F_0),\\
\cW h=F^{1/2}_0(G+F_0)^{-1/2}h,\; h\in\ran (G+F_0)^{1/2}.
\end{array}
\end{equation}
 Let $\overline\cW=\cW^{**}$ be the continuation of $\cW$ on $\cH$. Clearly, $\overline\cW^*=\cW^*.$
 Note that
 \[
 \cV^*\overline\cV h+\cW^*\overline \cW h=h,\;h\in\cran(G+F_0)
 \]
Set
\begin{equation}
\label{prosn}
\sN\stackrel{def}{=}\ker(I-\overline{\cW}\,\overline{\cW}^*).
\end{equation}
Since $\ker \cW^*=\ker F_0$, the subspace $\sN$ is contained in $\cran F_0$.
\begin{proposition}
\label{singar}
The equalities
\begin{multline}
\label{equival1}
\ran (I-\overline{\cW}\,\overline{\cW}^*)^{1/2}=\left\{f\in \cH: F^{1/2}_0f\in\ran G^{1/2}\right\}\\
=\left\{f\in \cH: F^{1/2}_0f\in\ran (F:G_0)^{1/2}\right\}
\end{multline}
hold.
\end{proposition}
\begin{proof}
Set $\cH_0\stackrel{def}{=}\cran(G+F_0)$. Note that $\ker (G+F_0)=\ker G\cap\ker F_0$.
Define
\begin{equation}
\label{mo}
M_0\stackrel{def}{=}\overline\cW^*\overline\cW\uphar\cH_0.
\end{equation}
Then $M_0\in\bB^+(\cH_0)$ and
\begin{equation}
\label{cvop}
\overline{\cV}^*\overline{\cV}\uphar=I_{\cH_0}-M_0=I_{\cH_0}-\overline\cW^*\overline\cW\uphar\cH_0.
\end{equation}
From \eqref{opercv} and \eqref{opwa}
\begin{multline}
\label{equiv11}
F^{1/2}_0f=G^{1/2}h\iff (G+F_0)^{1/2}\cW^*f=(G+F_0)^{1/2}\cV^* h\\
\iff \cW^*f=\cV^* h
\end{multline}
Equality \eqref{cvop} yields
\[
\ran \cV^*=\ran (I_{\cH_0}-\overline\cW^*\overline\cW\uphar\cH_0)^{1/2}
\]
Hence \eqref{equiv11} is equivalent to the inclusion $f\in \ran (I-\overline \cW\overline\cW^*)^{1/2}.$ Application of \eqref{ukfdyj} completes the proof.
\end{proof}
Thus from \eqref{prosn} and \eqref{equiv11} we get
\begin{equation}
\label{nob1}
\sN=\cH\ominus{\left\{{\rm clos}\left\{g\in\cH:F^{1/2}_0g\in\ran G^{1/2}\right\}\right\}}.
\end{equation}

\begin{theorem}
\label{form1}
Let $G\in\bB^+(\cH)$, $F_0\in\bB^+(\cH)$, $F_n\stackrel{def}{=}\mu_G(F_{n-1})$, $n\ge 1$,  $F\stackrel{def}{=}s-\lim_{n\to \infty}F_n$.
Then
\begin{equation}
\label{prosm1}
F=(G+F_0)^{1/2}P_\sM(G+F_0)^{1/2}
\end{equation}
and
\begin{equation}
\label{prosn1} F=F^{1/2}_0P_\sN F^{1/2}_0,
\end{equation}
where $\sM$ and $\sN$ are given by \eqref{prosm} and \eqref{nob1}, respectively.
\end{theorem}
\begin{proof}
From \eqref{contrw}, \eqref{contrw}, \eqref{mo}, \eqref{cvop}, \eqref{11} we have
\[
\begin{array}{l}
F_0=(G+F_0)^{1/2}M_0(G+F_0)^{1/2},\\
 G=(G+F_0)^{1/2}(I_{\cH_0}-M_0)(G+F_0)^{1/2},
\end{array}
\]
\[
\ker (I_{\cH_0}-M_0)=\sM,\;\cran (I-M_0)=\cH_0\ominus\sM=\Omega\ominus\ker (G+F_0).
\]
Then by Lemma \ref{yu1}
\begin{multline*}
F_0:G=(G+F_0)^{1/2}(M_0:(I_{\cH_0}-M_0))(G+F_0)^{1/2}\\
=(G+F_0)^{1/2}(M_0-M^2_0)(G+F_0)^{1/2}.
\end{multline*}
It follows that
\[
F_1=\mu_G(F_0)=F_0-F_0:G=(G+F_0)^{1/2}M^2_0(G+F_0)^{1/2}.
\]
Then (further $I=I_{\cH_0}$ is the identity operator) from \eqref{trans}
\begin{multline*}
F_1:G=(G+F_0)^{1/2}\left((I-M_0):M^2_0\right)(G+F_0)^{1/2}\\=(G+F_0)^{1/2}\left((I-M_0)M^2_0(I-M_0+M^2_0)^{-1}\right)(G+F_0)^{1/2},
\end{multline*}
\begin{multline*}
F_2\stackrel{def}{=}\mu_G(F_1)=F_1-F_1:G\\
=(G+F_0)^{1/2}\left(M_0^2-(I-M_0)M^2_0(I-M_0+M^2_0)^{-1}\right)(G+F_0)^{1/2}\\
=(G+F_0)^{1/2}M^4_0(I-M_0+M^2_0)^{-1}(G+F_0)^{1/2}.
\end{multline*}
Let us show by induction that for all $n\in\dN$
$$F_n\stackrel{def}{=}\mu_G(F_{n-1})=(G+F_0)^{1/2}M_n(G+F_0)^{1/2}\quad\mbox{for all} \quad n\in\dN,$$
 where
\begin{enumerate}
\item $\{M_n\}$ is a non-increasing sequence from $\bB^+(\cH_0)$,
\item $I-M_0+M_n$ is positive definite,
\item $M_n$ commutes with $M_0$,
\item $M_{n+1}=(I-M_0+M_n)^{-1}M^2_n.$
\end{enumerate}
All statements are already established for $n=1$ and for $n=2$. Suppose that all statements are valid for some $n$.
Further, using the equality $M_0M_n=M_nM_0$, we have
\begin{multline*}
I-M_0+M_{n+1}=I-M_0+(I-M_0+M_n)^{-1}M^2_n\\
=(I-M_0+M_n)^{-1}\left((I-M_0+M_n)(I-M_0)+M^2_n\right)\\
=(I-M_0+M_n)^{-1}\left((I-M_0)^2+M_n(I-M_0)+M^2_n\right)\\
=(I-M_0+M_n)^{-1}\left(\left((I-M_0)+\frac{1}{2}M_n\right)^2+\frac{3}{4}M^2_n\right).
\end{multline*}
Since
\[
(I-M_0)+\frac{1}{2}M_n\ge \frac{1}{2}\left(I-M_0+M_n\right),
\]
and $I-M_0+M_n$ is positive definite, we get that the operator $I-M_0+M_{n+1}$ is positive definite.
\begin{multline*}
M_0M_{n+1}=M_0(I-M_0+M_n)^{-1}M^2_n\\
=(I-M_0+M_n)^{-1}M^2_n M_0=M_{n+1}M_0.
\end{multline*}
From \eqref{trans} we have
\begin{multline*}
F_{n+1}:G=(G+F_0)^{1/2}\left((I-M_0):M_{n+1}\right)(G+F_0)^{1/2}\\
=(G+F_0)^{1/2}(I-M_0)M_{n+1}(I-M_0+M_{n+1})^{-1}(G+F_0)^{1/2},
\end{multline*}
and
\begin{multline*}
F_{n+2}=\mu_G(F_{n+1})=F_{n+1}-F_{n+1}:G\\
=(G+F_0)^{1/2}\left(M_{n+1}-(I-M_0)M_{n+1}(I-M_0+M_{n+1})^{-1}\right)(G+F_0)^{1/2}\\
=(G+F_0)^{1/2}(I-M_0+M_{n+1})^{-1}M^2_{n+1}(G+F_0)^{1/2}\\
=(G+F_0)^{1/2}M_{n+2}(G+F_0)^{1/2}.
\end{multline*}
One can prove by induction that inequality $I-M_n\ge 0$ and the equalities $M_{n+1}=(I-M_0+M_n)^{-1}M^2_n$ for all $n\in\dN$
imply
$$\ker(I-M_n)=\ker(I-M_0),\;n\in\dN.$$

Let $M=\lim\limits_{n\to\infty} M_n$. Then $F=(G+F_0)^{1/2}M(G+F_0)^{1/2}$. Since $M_{n+1}(I-M_0+M_{n})=M^2_n,$
we get
$(I-M_0)M=0$. Thus, $\ran M\subseteq\ker(I-M_0).$ Since $M\uphar\ker(I-M_0)=I$, we get
$M=P_{\ker(I-M_0)}.$ It follows that \eqref{prosm1} holds true.

The inequalities $0\le \mu_G(X)\le X$ yield $F_n=F^{1/2}_0N_nF^{1/2}_0,$
where $\{N_n\}$ is non-increasing sequence from $\bB^+(\cH)$, $0\le N_n\le I$ for all $n\in\dN$, and $\ker N_n\supseteq\ker F_0$. Let $ N=s-\lim_{n\to \infty}N_n$. Then
$F=F^{1/2}_0 NF^{1/2}_0$.
From \eqref{contrw} we have
\[
F^{1/2}_0=\cW(G+F_0)^{1/2}=(G+F_0)^{1/2}\cW^*,
\]
Since $M_0=\overline\cW^*\overline\cW\uphar\cH_0$ we get and $\overline\cW=VM^{1/2}_0$, where $V$ is isometry from $\cran M_0$ onto $\cran F_0$. Thus
\[
F^{1/2}_0=VM^{1/2}_0(G+F_0)^{1/2},\; M^{1/2}_0(G+F_0)^{1/2}=V^*F^{1/2}_0.
\]
Because $P_{\sM}=M^{1/2}_0P_{\sM} M^{1/2}_0$ we get from $F=(G+F_0)^{1/2}P_\sM(G+F_0)^{1/2}$:
\[
F=F^{1/2}_0VP_{\sM} V^*F^{1/2}_0.
\]
The operator $VP_{\sM} V^*$ is orthogonal projection in $\cran F_0$. Denote $\sN_0=\ran VP_{\sM} V^*=V\ran P_{\sM}.$
From $ (G+F_0)^{1/2}M^{1/2}_0h=F^{1/2}_0Vh$, for all $h\in\cran M_0$ we obtain
\[
(G+F_0)^{1/2}\f=F^{1/2}_0V\f,\; \f\in\sM=\ker(I_{\cH_0}-M_0),
\]
and then
\[
\f=(G+F_0)^{-1/2}F^{1/2}_0V\f.
\]
Hence
\[
(G+F_0)^{-1/2}F^{1/2}_0g=V^*g,\; g=V\f\in\sN_0.
\]
On the other hand
\[
(G+F_0)^{-1/2}F^{1/2}_0x=\overline{\cW}^*x\quad\mbox{for all}\quad x\in\cH.
\]
It follows that $\overline{\cW}^*g=V^*g$ for all $g\in\sN_0$. So
\[
g\in\sN_0\iff ||\overline{\cW}^*g||=||g||\iff g\in\ker (I-\overline{\cW}\,\overline{\cW}^*).
\]
Thus, $\sN_0$ coincides with $\sN$ defined in \eqref{prosn}, and \eqref{prosn1} holds true. \end{proof}
\begin{corollary}\label{commute}
Suppose $F_0$ commutes with $G$. Then $\sN$ defined in \eqref{prosn} takes the form $\sN=\ker G\cap\cran F_0.$
In particular,
\begin{enumerate}
\item if $\ker F_0\supseteq\ker G$, then $F=0$,
\item if $F_0=G$, then $F=0$,
\item  if $\ker G=\{0\}$, then $F=0$.
\end{enumerate}
\end{corollary}
\begin{proof}
If $F_0G=GF_0$. Then $F^{1/2}_0(G+F_0)^{-1/2}f=(G+F_0)^{-1/2}F^{1/2}_0f$ for all $f\in\ran (G+F_0)^{1/2}$. Hence, $\cW^*=\overline\cW=\cW^{**}$ and $\overline \cW$  is nonnegative contraction. It follows from \eqref{prosn} that
\[
\sN=\ker(I-\overline{\cW}^2)=\ker (I-\cW^*)
=\ker (I-(G+F_0)^{1/2}F_0^{1/2}).
\]
Clearly
\[
f\in \ker (I-(G+F_0)^{1/2}F_0^{1/2})\iff f\in\ker G\cap\cran F_0.
\]
Furthermore, applying \eqref{prosn1} we get implications
\[
\begin{array}{l}
\ker F_0\supseteq\ker G\Longrightarrow \sM_0=\{0\},\\
\ker G=\{0\}\Longrightarrow\sM=\{0\}.
\end{array}
\]
\end{proof}
\begin{corollary}
\label{new1}
If $G\in\bB^+_0(\cH)$ and if $F_0$ is positive definite, then $F=0$.

\end{corollary}
\begin{proof}
In the case when $F_0$ is positive definite the subspace $\sM$ defined in \eqref{prosm} can be described as follows:
$\sM= (G+F_0)^{1/2}\ker G$. Hence,  if $\ker G=\{0\}$, then $\sM=\{0\}$ and \eqref{prosm1} gives $F=0$.

\end{proof}

\begin{theorem} Let $G\in\bB^+(\cH)$, $F_0 \in \bB^+(\cH)$, $F_{n+1}=\mu_G(F_n)$, $n\ge 0$, $F=\lim_{n\to \infty}F_n$.
\begin{enumerate}
\item If $\ran F^{1/2}_0\subseteq\ran G^{1/2}$, then $F=0$.
\item If $\ran F^{1/2}_0=\ran G^{\alpha}$, where $\alpha<1/2$, then $F=0$.
\end{enumerate}
\end{theorem}
\begin{proof}
(1) Let $\ran F^{1/2}_0\subseteq\ran G^{1/2}$. Then
$
F^{1/2}_0\cH\subseteq\ran G^{1/2}
$. From \eqref{nob1} and \eqref{prosn1} it follows $F=0$.

(2) Suppose $\ran F^{1/2}_0=\ran G^{\alpha},$ where $\alpha< 1/2$.  Then by Douglas theorem \cite{Doug} the operator $F_0$ is of the form
\[
F_0=G^{\alpha}Q_0G^{\alpha},
\]
where $Q$ is positive definite in $\cH_0=\cran G$.
Hence, $G+G^{\alpha}QG^{\alpha}=G^{\alpha}(G^{1-2\alpha}+Q_0)G^{\alpha}$, and
\begin{multline*}
\mu_G(F_0)= \left((G+G^{\alpha}Q_0G^{\alpha})^{-1/2}G^{\alpha}Q_0G^{\alpha} \right)^*(G+G^{\alpha}Q_0G^{\alpha})^{-1/2}G^{\alpha}Q_0G^{\alpha}\\
=G^{\alpha}Q_0(G^{1-2\alpha}+Q_0)^{-1}Q_0G^{\alpha}=G^{\alpha}\mu_{G^{1-2\alpha}}(Q_0)G^{\alpha}.
\end{multline*}
Note that $Q_1\stackrel{def}{=}\mu_{G^{1-2\alpha}}(Q_0)$ is positive definite. Therefore for $F_1=\mu_G(F_0)$ possess the property
$\ran F^{1/2}_1=\ran G^{\alpha}$. By induction we can prove that
\[
F_{n+1}=\mu_G(F_n)=G^{\alpha}\mu_{G^{1-2\alpha}}(Q_n)G^{\alpha}
=G^{\alpha}Q_{n+1}G^{\alpha}.
\]
Using that $Q_0$ is positive definite and applying Corollary \ref{new1}, we get $\lim_{n\to\infty}Q_n=0$. Hence
\[
F=\lim\limits_{n\to\infty}F_n=\lim\limits_{n\to\infty}G^{\alpha}Q_nG^{\alpha}=0.
\]
\end{proof}

\begin{corollary}
\label{mnogo}
Let $\lambda>0$. Define a subspace
\[
\sM_\lambda=\cH\ominus{\left\{{\rm clos}\left\{g\in\cH:(\lambda G+F_0)^{1/2}g\in\ran G^{1/2}\right\}\right\}}
\]
Then
\[
(\lambda G+F_0)^{1/2}P_{\sM_\lambda}(\lambda G+F_0)^{1/2}
=F^{1/2}_0P_{\sN} F^{1/2}_0,
\]
where $\sN$ is given by \eqref{nob1}.

\end{corollary}
\begin{proof} Replace $G$ by $\lambda G$ and consider a sequence
$$F_0, F_1=\mu_{\lambda G}(F_0),\;F_{n}=\mu_{\lambda G}(F_{n-1}),\ldots.$$
Clearly
\begin{multline*}
\cH\ominus{\left\{{\rm clos}\left\{g\in\cH:F^{1/2}_0g\in\ran (\lambda G)^{1/2}\right\}\right\}}\\=
\cH\ominus{\left\{{\rm clos}\left\{g\in\cH:F^{1/2}_0g\in\ran G^{1/2}\right\}\right\}}=\sN.
\end{multline*}
By Theorem \ref{form1}
\[
s-\lim\limits_{n\to\infty}F_n=F^{1/2}_0P_\sN F^{1/2}_0.
\]
On the other side the application of \eqref{prosm1} gives
\[
s-\lim\limits_{n\to\infty}F_n=(\lambda G+F_0)^{1/2}P_{\sM_\lambda}(\lambda G+F_0)^{1/2}.
\]

 \end{proof}

\begin{theorem}
\label{interzero}
Let $G\in\bB^+_0(\cH)$, $\ran G\ne \cH$. Let $F_0\in\bB^+(\cH)$, $F_n\stackrel{def}{=}\mu_G(F_{n-1})$, $n\ge 1$,  $F\stackrel{def}{=}s-\lim_{n\to \infty}F_n$.
Then
\begin{multline*}
F\in\bB^+_0(\cH)\Longrightarrow \left\{\begin{array}{l}F_0\in\bB^+_0(\cH),\\
\ran(G+F_0)\cap\ran G^{1/2}=\{0\}\end{array}\right.\\
\iff
\left\{\begin{array}{l}F_0\in\bB^+_0(\cH),\\
\ran F_0\cap\ran G^{1/2}=\{0\}\end{array}\right..
\end{multline*}
Moreover, the following conditions are equivalent:
\begin{enumerate}
\def\labelenumi{\rm (\roman{enumi})}
\item $F\in\bB^+_0(\cH)$,
\item $\ran (G+F_0)^{1/2}\cap \cran (G+F_0)^{-1/2}G^{1/2}=\{0\}$,
\item for each converging sequence $\{y_n\}\subset\ran G^{1/2}$ such that
$$\lim_{n\to\infty}y_n\in\ran F_0$$
follows that the sequence
$\{(G+F_0)^{-1/2}y_n\}$ is diverging,
\item $\ran F^{1/2}_0\cap\clos\left\{F^{-1/2}_0\left(\ran F^{1/2}_0\cap \ran G^{1/2}\right)\right\}=\{0\}$,
\item for each converging sequence $\{z_n\}\subset \ran F^{1/2}_0\cap\ran G^{1/2}$ such that
$$\lim_{n\to\infty}z_n\in\ran F_0$$
follows that the sequence $\{F^{-1/2}_0 z_n\}$ is diverging.
\end{enumerate}
\end{theorem}
\begin{proof}
 Clearly $F\in\bB^+_0(\cH)\iff\ker F=\{0\}$. Since $\ker (G+F_0)=\{0\},$ from \eqref{prosm1}, \eqref{prosm}, \eqref{contrv}, \eqref{opercv}  it follows equivalences
\begin{multline*}
\ker F=\{0\}\iff\Omega\cap \ran (G+F_0)^{1/2}=\{0\}\\
\iff\ran (G+F_0)^{1/2}\cap \cran (G+F_0)^{-1/2}G^{1/2}=\{0\}.
\end{multline*}
So (i)$\iff$(ii).
In particular
$$ \ker F=\{0\}\Longrightarrow \ran (G+F_0)^{1/2}\cap \ran (G+F_0)^{-1/2}G^{1/2}=0.$$
Hence
\begin{equation}
\label{inters}
\ran (G+F_0)\cap \ran  G^{1/2}=0.
\end{equation}
Assume that $\ran G^{1/2}\cap\ran F_0\ne\{0\}$. Then $F_0x=G^{1/2}y$
for some $x,y\in \cH$. Set $z\stackrel{def}{=}y+G^{1/2}x$. Then $F_0x=G^{1/2}(z-G^{1/2}x)$ and $(G+F_0)x=G^{1/2}z$ that  contradicts to \eqref{inters}.

Conversely, if $\ran (G+F_0)\cap \ran  G^{1/2}\ne\{0\}$, then $\ran G^{1/2}\cap\ran F_0\ne\{0\}$.
So, \eqref{inters} is equivalent to $\ran G^{1/2}\cap\ran F_0=\{0\}.$
Note that the latter is equivalent to $F^2_0:G=0.$

Suppose $\ran (G+F_0)^{1/2}\cap \cran (G+F_0)^{-1/2}G^{1/2}\ne\{0\}.$ Then there is a sequence $\{x_n\}\subset \cH$ and a vector $f\in\cH$ such that
\[
(G+F_0)^{1/2}f=\lim\limits_{n\to\infty}(G+F_0)^{-1/2}G^{1/2}x_n
\]
Hence $\lim\limits_{n\to\infty}G^{1/2}x_n=(G+F_0)f.$
Let $y_n=G^{1/2}(x_n-G^{1/2}f),$ $n\in\dN$. Then $\{y_n\}\subset\ran G^{1/2}$, $\lim\limits_{n\to\infty}y_n=F_0f,$
and
\[
\lim\limits_{n\to\infty}(G+F_0)^{-1/2}y_n=(G+F_0)^{1/2}f-(G+F_0)^{-1/2}Gf.
\]
Conversely, if there is converging sequence $\{y_n=G^{1/2}z_n\}$ such that
$$\lim_{n\to\infty}y_n= F_0f$$
and the sequence $\{(G+F_0)^{-1/2}y_n\}$ converges as well, then  from
$$\lim_{n\to\infty}G^{1/2}(z_n+G^{1/2}f)=(G+ F_0)f$$
and because the operator $(G+F_0)^{-1/2}$ is closed, we get
\begin{multline*}
(G+F_0)^{1/2}f=(G+F_0)^{-1/2}(G+F_0)f\\
=\lim\limits_{n\to\infty}(G+F_0)^{-1/2}G^{1/2}(z_n+G^{1/2}f).
\end{multline*}
This means that
$\ran (G+F_0)^{1/2}\cap \cran (G+F_0)^{-1/2}G^{1/2}\ne\{0\}.$ Thus, conditions (i) and (ii) are equivalent.
 Using \eqref{ukfdyj}, \eqref{contrw}, \eqref{opwa}, \eqref{equival1}, \eqref{prosn1}, and
Theorem \ref{form1}, the equivalences (i)$\iff$(iv)$\iff$(v) can be proved similarly.
\end{proof}

\section{The mapping $\tau_G$}
Recall that the mapping $\mu_G$ is defined by \eqref{mapmu} and by $\mu^{[n]}_G$ we denote the $n$th iteration of the mapping $\mu_G$.
Note that
\[
\mu^{[n+1]}_G(X)=\mu^{[n]}_G(X)-\mu^{[n]}_G(X):G,\; n\ge 0.
\]
Hence
\begin{equation}
\label{rec}
\sum\limits_{k=0}^n\left(\mu^{[k]}_G(X):G\right)=X-\mu^{[n+1]}_G(X).
\end{equation}
Clearly
\[
X\ge \mu_G(X)\ge \mu^{[2]}_G(X)\ge\cdots\ge \mu^{[n]}_G(X)\ge\cdots.
\]
Therefore, the mapping
\[
\bB^+(\cH)\ni X\mapsto\tau_G(X)\stackrel{def}{=}s-\lim\limits_{n\to\infty}\mu^{[n]}_G(X)\in\bB^+(\cH)
\]
is well defined. Besides, using \eqref{rec} and the monotonicity of parallel sum, we see that
\begin{enumerate}
\item $ \mu^{[n]}_G(X):G\ge \mu^{[n+1]}_G(X):G$ for all $n\in\dN_0,$
\item the series $\sum\limits_{n=0}^\infty \left(\mu^{[n]}_G(X):G\right)$ is converging in the strong sense and
\begin{equation}
\label{ryad}
\sum\limits_{n=0}^\infty \left(\mu^{[n]}_G(X):G\right)=X-\tau_G(X).
\end{equation}
\end{enumerate}
Hence the mapping $\tau_G$ can be defined as follows:
\[
\tau_G(X)\stackrel{def}{=}X-\sum\limits_{n=0}^\infty \left(\mu^{[n]}_G(X):G\right).
\]

Most of the following properties of the mapping $\tau_G$ are already established in the statements above.
\begin{theorem}
\label{propert}
The mapping $\tau_G$ possesses the properties:
\begin{enumerate}
\item $\tau_G(\mu_G(X))=\tau_G(X)$ for all $X\in\bB^+(H),$ therefore,\\ $\tau_G(\mu_G^{[n]}(X))=\tau_G(X)$ for all natural $n$;
 \item  $\tau_G(X):G=0$ for all $X\in\bB^+(H);$
\item  $\tau_G(X)\le X$ for all $X\in\bB^+(\cH)$ and $\tau_G(X)=X$ $\iff$ $X:G=0$ $\iff$ $\ran X^{1/2}\cap\ran G^{1/2}=\{0\}$;
\item $\tau_G(X)=\tau_G(\tau_G(X))$ for an arbitrary $X\in\bB^+(\cH)$;
\item define a subspace
\begin{equation}
\label{ghjcn1}
\sM:
=\cH\ominus{\rm{clos}}\left\{f\in\cH,\;(G+X)^{1/2}f\in\ran G^{1/2}\right\},
\end{equation}
then
\begin{equation}
\label{formula11}
\tau_G(X)=(G+X)^{1/2}P_\sM(G+X)^{1/2};
\end{equation}
\item define a contraction $\cT=(G+X)^{-1/2}X^{1/2}$
and subspace
 \[
\sL\stackrel{def}{=}\ker(I-\cT^*\cT),
\]
then
\begin{equation}
\label{ghjcn2}
\sL=\cH\ominus{\left\{{\rm clos}\left\{g\in\cH,\;X^{1/2}g\in\ran G^{1/2}\right\}\right\}}
\end{equation}
and
\begin{equation}
\label{formula111}
\tau_G(X)=X^{1/2}P_\sL X^{1/2};
\end{equation}
in particular, if $X$ is positive definite, then $\sL=X^{1/2}\ker G$;
\item $XG=GX$ $\Longrightarrow$ $\tau_G(X)=X^{1/2}P_{\sN}X^{1/2}$, where $\sN$ takes the form $\sN=\ker G\cap\cran X$;
\item $\tau_G(G)=0$;
\item $\ran X^{1/2}\subseteq\ran G^{1/2}$ $\Longrightarrow$ $\tau_G(X)=0;$ in particular,
$$\tau_G\left(X:G\right)=0$$
 for every $X\in\bB^+(\cH)$;
\item $\ran X^{1/2}=\ran G^{\alpha}$, $\alpha<1/2$ $\Longrightarrow$ $\tau_G(X)=0;$
\item $\tau_G(\lambda G+X)=\tau_{\eta G}(X)=\tau_G(X)$ for all $\lambda>0$ and $\eta>0$;
\item $\tau_G(\xi X)=\xi \tau_G(X),$ $\xi>0;$
\item if $\ran G^{1/2}_1=\ran G^{1/2}_2$, then
\[
\tau_{G_1}(X)=\tau_{G_2}(X)=
\tau_{G_1}(G_2+X)=\tau_{G_2}(G_1+X)
\]
for all $X\in\bB^+(\cH)$;
\item if $\ran G^{1/2}_1\subseteq\ran G^{1/2}_2$, then $\tau_{G_1}(X)\ge \tau_{G_2} (X)$ for all $X\in\bB^+(\cH)$;
\item $\tau_G(X)\in \bB^+_0(\cH)$ $\Longrightarrow$ $X\in\bB^+_0(\cH)$ and $X^2:G=0$;
\item the following conditions are equivalent:
\begin{enumerate}
\def\labelenumi{\rm (\roman{enumi})}
\item $\tau_G(X)\in\bB^+_0(\cH)$,
\item $ X\in \bB^+_0(\cH)$ and $\ran (G+X)^{1/2}\cap\clos\{(G+X)^{-1/2}\ran G^{1/2}\}=\{0\}$,
\item $ X\in \bB^+_0(\cH)$  and for each converging sequence $\{y_n\}\subset\ran G^{1/2}$ such that
$$\lim_{n\to\infty}y_n\in\ran X$$ it
follows that the sequence
$\{(G+X)^{-1/2}y_n\}$ is diverging,
\item $ X\in \bB^+_0(\cH)$ and
$\ran X^{1/2}\cap\clos\left\{X^{-1/2}\left(\ran X^{1/2}\cap \ran G^{1/2}\right)\right\}=\{0\}$,
\item$ X\in \bB^+_0(\cH)$ and for each converging sequence $\{z_n\}\subset \ran X^{1/2}\cap\ran G^{1/2}$ such that
$$\lim_{n\to\infty}z_n\in\ran X$$
follows that the sequence $\{X^{-1/2} z_n\}$ is diverging;
\end{enumerate}
\item and if $X$ is a compact operator, then $X$ is a compact operator as well, moreover, if $\tau_G(X)$ from the Shatten-von Neumann class $S_p$ \cite{GK}, then $\tau_G(X)\in S_p.$

 \end{enumerate}
\end{theorem}
\begin{proof}
Equalities in (6) follow from \eqref{contrw}, Proposition \ref{singar} and Theorem \ref{form1}, (11) follows from Corollary \ref{mnogo}.
If $\xi>0$, then
\begin{multline*}
\tau_G(\xi X)=(G+\xi X)^{1/2}P_{\sM_{1/\xi}} (G+\xi X)^{1/2}\\
=\xi ((1/\xi)G+X)^{1/2}P_{\sM_{1/\xi}} ((1/\xi)G+X)^{1/2}\\
=\xi\tau_G(X).
\end{multline*}
This proves (12).

If $\ran G^{1/2}_1=\ran G^{1/2}_2$, then
\[
X^{1/2}g\in\ran G^{1/2}_1\iff X^{1/2}g\in\ran G^{1/2}_2.
\]
Now from property (6) follows  the equality $\tau_{G_1}(X)=\tau_{G_2}(X)$. Using (11) we get
\begin{multline*}
\tau_{G_1}(G_2+X)=\tau_{G_2}(G_2+X)=\tau_{G_2}(X)\\
=\tau_{G_1}(X)=\tau_{G_1}(G_1+X)=\tau_{G_2}(G_1+X).
\end{multline*}
 So, property (13) is proved. If $\ran G^{1/2}\subseteq\ran G^{1/2}_2$, then
 $$X^{1/2}g\in\ran G^{1/2}_1\Longrightarrow X^{1/2}g\in\ran G^{1/2}_2.$$
Hence
\begin{multline*}
\sL_1=\cH\ominus{\left\{{\rm clos}\left\{g\in\cH:X^{1/2}g\in\ran G^{1/2}_1\right\}\right\}}\\
\supseteq \sL_2=\cH\ominus{\left\{{\rm clos}\left\{g\in\cH:X^{1/2}g\in\ran G^{1/2}_2\right\}\right\}},
\end{multline*}
and
\[
\tau_{G_1}(X)=X^{1/2}P_{\sL_1}X^{1/2}\ge X^{1/2}P_{\sL_2}X^{1/2}=\tau_{G_2}(X).
\]
If $X$ is compact operator, then from $\tau_G(X)=X^{1/2}P_\sL X^{1/2}$ it follows that $\tau_G(X)$ is compact operator. If $X\in S_p$, where $p\ge 1$ and $S_p$ is Shatten--von Neumann ideal, then from $X^{1/2},P_\sL X^{1/2}\in S_{2p}$ follows that $X^{1/2}P_\sL X^{1/2}\in S_p$ \cite[page 92]{GK}.
\end{proof}
\begin{remark}
Given $G\in\bB^+(\cH)$. All $ \wt G\in\bB^+(\cH)$ such that $\ran \wt G^{1/2}=\ran G^{1/2}$ are of the form
\[
\wt G=G^{1/2}Q G^{1/2},
\]
where $Q,Q^{-1}\in\bB^+(\cran G)$.
\end{remark}
\begin{remark}
\label{extr}
Let $G,\wt G\in\bB^+(\cH)$ and $\ran G^{1/2}=\ran\wt G^{1/2}$.
 The equalities
 $$\tau_G(\wt G+X)=(\wt G+X)^{1/2}\wt P(\wt G+X)^{1/2}=\tau_G(X)=X^{1/2}P_\sL X^{1/2},$$
 where $\wt P$ is the orthogonal projection onto the subspace
 \[
 \cH\ominus{\rm{clos}}\left\{f\in\cH:(\wt G+X)^{1/2}f\in\ran G^{1/2}\right\},
 \]
 see \eqref{ghjcn1} and \eqref{ghjcn2},
  show that $\tau_G (X)$ is an extreme point of the operator interval $[0,X]$ and operator intervals $[0, \wt G+X]$  cf. \cite{Ando_1996}.
\end{remark}
\begin{remark}
\label{osta}
Let $G,X\in\bB_0^+(\cH)$, $\ran G^{1/2}\cap \ran X^{1/2}=\{0\}$. From properties (13) and (16) in Theorem \ref{propert} follows that if the equality
\[
\ran (G+X)^{1/2}\cap\cran((G+X)^{-1/2}G^{1/2})=\{0\}
\]
holds true, then it remains valid if $G$ is replaced  by $\wt G$ such that $\ran \wt G^{1/2}=\ran G^{1/2}.$
\end{remark}

\begin{proposition}
\label{polez}
1) Assume $G\in \bB^+(\cH)$. (a) If $X:G\ne 0,$  then $\left(\mu^{[n]}_G(X)\right):G\ne 0$ for all $n$.

b) If $X\in\bB^+_0(\cH)$, then $\mu^{[n]}_G(X)\in\bB^+_0(\cH)$ for all $n$. Moreover, if $\ran X^{1/2}\supseteq\ran G^{1/2},$ then
$\ran\left(\mu^{[n]}_G(X)\right)^{1/2}=\ran X^{1/2}$ for all $n$.\\
\noindent 2) If $G\in\bB^+_0(\cH)$ and $\tau_G(X)\in\bB^+_0(\cH)$, then $\mu^{[n]}_G(X)\in\bB^+_0(\cH)$
\begin{equation}
\label{dobav2}
\ran \left(\mu^{[n]}_G(X)\right)^{1/2}\cap\clos\left\{\left(\mu^{[n]}_G(X)\right)^{-1/2}\ran G^{1/2}\right\}=\{0\},
\end{equation}
in particular, $\left(\mu^{[n]}_G(X)\right)^2:G=0$ ($\iff \ran\mu^{[n]}_G(X)\cap\ran G^{1/2}=\{0\}$) for all $n$.
\end{proposition}
\begin{proof} Due to the property $\tau_G(\mu_G(X))=\tau_G(X)$ for all $X\in\bB(\cH)$,
it is sufficient to prove that the assertions of proposition hold for $n=1$. Let $\cH_0=\cran(G+X)$.
There exists $M\in \bB^+(\cH_0)$ such that
\[
X=(G+X)^{1/2}M(G+X)^{1/2}, \; G=(G+X)^{1/2}(I-M)(G+X)^{1/2}.
\]
Then
\begin{multline*}
\mu_G(X)=X-X:G\\=(G+X)^{1/2}M(G+X)^{1/2}-(G+X)^{1/2}M(I-M)(G+X)^{1/2}\\
=(G+X)^{1/2}M^2(G+X)^{1/2}.
\end{multline*}
It follows
\[
\ran\left(\mu_G(X)\right)^{1/2}=(G+X)^{1/2}\ran M.
\]
Because $X:G\ne 0$, we have $\ran X^{1/2}\cap\ran G^{1/2}\ne\{0\}$. Therefore
\[
\ran M^{1/2}\cap\ran (I-M)^{1/2}\ne \{0\}.
\]
This means that there are $f,h\in\cH$ such that $M^{1/2}f=(I-M)^{1/2}h$. Hence
\[
Mf=(I-M)^{1/2}M^{1/2}h.
\]
Since $\ran(X:G)^{1/2}=(G+X)^{1/2}\ran (M-M^2)^{1/2}$, we get
$$\ran\left(\mu_G(X)\right)^{1/2}\cap\ran(X:G)^{1/2}\ne \{0\}.$$
But $\ran(X:G)^{1/2}\subseteq\ran G^{1/2}$. Hence $\mu_G(X):G\ne 0.$

Clearly
\begin{multline*}
\ran X^{1/2}\supseteq\ran G^{1/2}\iff \ran M^{1/2}\supseteq\ran (I-M)^{1/2}\\
\iff\ran  M=\cH_0.
\end{multline*}
Hence
\begin{multline*}
\ran\left(\mu_G(X)\right)^{1/2}=(G+X)^{1/2}\ran M=\ran (G+X)^{1/2}\\
=\ran X^{1/2}\supseteq\ran G^{1/2}.
\end{multline*}

If $\ker X=\{0\}$, then $\ker(G+X)=\{0\}$ and $\ran(G+X)^{1/2}\cap\ker M =\{0\}$. It follows that $\ran(G+X)^{1/2}\cap\ker M^2 =\{0\}$.
Hence $\ker\mu_G(X)=\{0\}$.

Since $\tau_G(\mu_G(X))=\tau_G(X)$ and $\tau_G(X)\in\bB^+_0(\cH)$ implies $\ker X=\{0\}$ and $X^2:G=0$, see Theorem \ref{interzero}, we get
$$\tau_G(X)\in\bB^+_0(\cH)\Longrightarrow \ker\mu_G(X)=\{0\},\; \left(\mu_G(X)\right)^2:G=0.$$
\end{proof}
\begin{remark}
\label{dobav} Let $G\in\bB^+_0(\cH)$. Assume that $\ran X^{1/2}\supset \ran G^{1/2}$ and $\tau_G(X)\in\bB^+_0(\cH)$.
Denoting $\sM_n=\clos\left\{\left(\mu^{[n]}_G(X)\right)^{-1/2}\ran G^{1/2}\right\}$, one obtains from \eqref{dobav2} that
\[
\sM_n\cap\ran X^{1/2}=\sM_n^\perp\cap\ran X^{1/2}=\{0\}\;\forall n\in\dN.
\]
These relations yield
\[
\sM_n\cap\ran G^{1/2}=\sM_n^\perp\cap\ran G^{1/2}=\{0\}\;\forall n\in\dN.
\]
If $J_n=P_{\sM_n}-P_{\sM^\perp_n}=2P_{\sM_n}-I,$ $n\in\dN$, then $J_n=J_n^*=J^{-1}_n$ ($J_n$ is a fundamental symmetry in $\cH$ for each  natural number $n$), and
\[
\ran (J_nG^{1/2}J_n)\cap\ran G^{1/2}=\{0\}\; \forall n\in\dN,
\]
cf. \cite{Arl_ZAg_IEOT_2015}, \cite{schmud}.
\end{remark}
Let $G\in\bB^+(\cH)$. Set
\[
\bB^+_G(\cH)=\left\{Y\in\bB^+(\cH): \ran Y^{1/2}\cap\ran G^{1/2}=\{0\}\right\}.
\]
Observe that $Y\in\bB^+_G(\cH)\Longrightarrow Y^{1/2}QY^{1/2}\in \bB^+_G(\cH)$ for an arbitrary $Q\in\bB^+(\cH)$.
The cone $\bB^+_G(\cH)$ is the set of all fixed points of the mappings $\mu_G$ and $\tau_G$. In addition
\[
\bB^+_G(\cH)=\tau_G(\bB^+(\cH)).
\]
Actually, property (13) in Theorem \ref{propert} shows that
if $Y\in\bB^+_G(\cH)$, then for each $\wt G\in\bB^+(\cH)$ such that $\ran \wt G^{1/2}=\ran G^{1/2}$, the operator $Y+\wt G$ is contained in the pre-image $\tau_G^{-1}\{Y\}$,
i.e., the equality
\[
\tau_G(\wt G+Y)=Y=(\wt G+Y)^{1/2}P_{\sM_{\wt G }}(\wt G+Y)^{1/2}
\]
holds, where
$$\sM_{\wt G}=\cH\ominus\{g\in\cH:(\wt G+Y)^{1/2}g\in\ran G^{1/2}\}.$$
In particular,
 \[
 \tau_G(\wt G+\tau_G(X))=\tau_G(X),\;\forall X\in\bB^+(\cH).
 \]
  Thus, the operator $\wt G+Y$ is contained in the \textit{basin of attraction} of the fixed point $Y$ of the mapping $\mu_G$ for an arbitrary $\wt G\in\bB^+(\cH)$ such that $\ran \wt G^{1/2}=\ran G^{1/2}$. In addition since $\ran(\wt G+Y)^{1/2}=\ran G^{1/2}\dot+\ran Y^{1/2}$, the statement 1 b) of Proposition  \ref{polez} yields that
 $$\ran \left(\mu^{[n]}_G(\wt G+Y)\right)^{1/2}=const\supset\ran G^{1/2}\;\forall n\in\dN.$$

\section{Lebesgue type decomposition of nonnegative operators and the mapping $\tau_G$}
Let $A\in\bB^+(\cH)$. T.~Ando in \cite{Ando_1976} introduced and studied the mapping
\[
\bB^+(\cH)\ni B\mapsto [A]B\stackrel{def}{=}s-\lim\limits_{n\to\infty}(nA:B)\in\bB^+(\cH).
\]
The decomposition
\[
B=[A]B+(B-[A]B)
\]
provides the representation of $B$ as the sum of $A$-\textit{absolutely continuous} ($[A]B$)  and $A$-\textit{singular} ($(B-[A]B$) parts  of $B$ \cite{Ando_1976}.
An operator $C\in\bB^+(\cH)$ is called $A$-absolutely continuous \cite{Ando_1976} if there exists a nondecreasing sequence $\{C_n\}\subset\bB^+(\cH)$
such that $C=s-\lim_{n\to\infty}C_n$ and $C_n\le \alpha_n A$ for some $\alpha_n$, $n\in\dN$ ($\iff \ran C^{1/2}_n\subseteq\ran A^{1/2}$ $\forall n\in\dN$). An operator $C\in\bB^+(\cH)$ is called $A$-singular if the intersections of operator intervals $[0,C]$ and $[0,A]$ is the trivial operator ($[0,C]\cap [0,A]=0$). Moreover, the operator $[A]B$ is maximum among all $A$-absolutely continuous nonnegative operators $C$ with $C\le B$.
The decomposition of $B$ on $A$-absolutely continuous and $A$-singular parts is generally non-unique. Ando in \cite{Ando_1976} proved that uniqueness holds if and only if $\ran([A]B)^{1/2}\subseteq\ran A^{1/2}$.
Set
\begin{equation}
\label{omtuf}
\Omega_{A}^B\stackrel{def}{=}{\rm{clos}}\left\{f\in\cH:B^{1/2}f\in\ran A^{1/2}\right\}.
\end{equation}
It is established in \cite{Ando_1976} that the following conditions are equivalent
\begin{enumerate}
\def\labelenumi{\rm (\roman{enumi})}
\item $B$ is $A$-absolutely continuous,
\item $[A]B=B,$
\item $\Omega_A^B=\cH$.
\end{enumerate}
In \cite{Pek_1978} (see also \cite{Kosaki_1984}) the formula
\begin{equation}
\label{formu}
[A]B=B^{1/2}P_{\Omega^B_{A}}B^{1/2}
\end{equation}
has been established.
Hence the operator $[A]B$ possesses the following property, see \cite{Pek_1978}:
\begin{multline*}
\max\left\{Y\in\bB^+(\cH):0\le Y\le B,\;{\rm{clos}}\{Y^{-1/2}(\ran A^{1/2})\}=\cH\right\}\\
=[A]B.
\end{multline*}
The notation $B_{\ran A^{1/2}}$ and the name \textit{convolution on the operator domain} was used for $[A]B$ in \cite{Pek_1978}.
Notice that from \eqref{formu} it follows the equalities
\begin{multline*}
\ran\left([A]B\right)^{1/2}=B^{1/2}\Omega_A^B,\\
B-[A]B=B^{1/2}(I-P_{\Omega^B_{A}})B^{1/2},\\
[A]B:(B-[A]B)=0,\; A:(B-[A]B)=0.
\end{multline*}
In addition due to \eqref{ukfdyj}, \eqref{omtuf}, and \eqref{formu}:
\begin{enumerate}
\item $[A](\lambda B)=\lambda\left([A]B\right),$ $\lambda>0,$
\item $\ran \wt A^{1/2}=\ran A^{1/2}$ $\Longrightarrow$ $[\wt A]B=[A]B $ for all $B\in\bB^+(\cH),$
\item $[A:B]B=[A]B$.
\end{enumerate}

\begin{theorem}
\label{singp}
\begin{enumerate}
\item
 Let $G\in\bB^+(\cH)$. Then for each $X\in\bB^+(\cH)$ the equality
\begin{equation}
\label{razn}
\tau_G(X)=X-[G]X
\end{equation}
holds. Therefore, $\tau_G(X)=0$ if and only if $X$ is $G$-absolutely continuous. In addition
$\tau_G([G]X)=0$ for all $X\in\bB^+(\cH)$.
If $\ran \wt G^{1/2}=\ran G^{1/2}$ for some $\wt G \in\bB^+(\cH)$, then
\begin{equation}
\label{cbyuek}
\tau_G(X)=X-[\wt G] X=\wt G+X-[G](\wt G+X).
\end{equation}
Hence
\begin{equation}
\label{cnhfyyj}
\wt G=[G](\wt G +X)-[G](X),
\end{equation}
and
\begin{equation}
\label{tot}
X-\tau_G(X)=[G](\wt G+X)-\wt G.
\end{equation}
In addition
\begin{equation}
\label{izm}
\sum\limits_{n=0}^\infty \left(\mu^{[n]}_G(X):G\right)=[G]X,\; \forall X\in\bB^+(\cH).
\end{equation}
\item The following inequality is valid for an arbitrary $X_1, X_2\in\bB^+(\cH)$:
\begin{equation}
\label{ytjblf}
\tau_G(X_1+X_2)\le \tau_G(X_1)+\tau_G(X_2).
\end{equation}

\item  the following statements are equivalent:
\begin{enumerate}
\def\labelenumi{\rm (\roman{enumi})}
\item $\tau_G(X)\in\bB^+_0(\cH)$,
\item $X\in\bB^+_0(\cH)\quad\mbox{and}\quad\left([G]X\right):X^2=0,$
\item $G+X\in\bB^+_0(\cH)$ and $[G](G+X):( G+X)^{2}=0.$
\end{enumerate}
\end{enumerate}
\end{theorem}
\begin{proof}
(1) From \eqref{omtuf}, \eqref{formu}, and Theorem \ref{propert} we get equalities
\begin{multline*}
\tau_G(X)=X^{1/2}(I-P_{\Omega^X_{G}})X^{1/2}=X-[G]X,\\
\tau_G(\wt G+X)=(\wt G+X)^{1/2}(I-P_{\Omega^{\wt G+X}_{G}})(\wt G+X)^{1/2}=\wt G+X-[G](\wt G+X).
\end{multline*}
Then \eqref{cbyuek}, \eqref{cnhfyyj}, and \eqref{tot} follow from the equalities $\tau_G(X)= \tau_{\wt G}[X]=\tau_G(\wt G+X)$.
Since $[G]([G]X)=[G]X$, we get  $\tau_G\left([G]X\right)=0.$

Note that using the equality $[G](X+\alpha G)=[G]X+\alpha G$ \cite[Lemma 1]{Nishio}
 and the equality
 \[
 \tau_G(\alpha G+X)=\alpha G+X-[G](\alpha G+X),
 \]
 we get $\tau_G(\alpha G+X)=X-[G]X=\tau_G(X)$.

Equation \eqref{izm} follows from \eqref{ryad} and \eqref{razn}.

(2) Inequality \eqref{ytjblf} follows from the inequality, see \cite{E-L},
$$[G](X_1+X_2)\ge [G]X_1+[G]X_2$$
 and equality \eqref{razn}.

(3) From \eqref{omtuf} and statements (16a) and (16d) of Theorem \ref{propert} it follows
\begin{multline*}
\tau_G(X)\in\bB^+_0(\cH)\iff X\in\bB^+_0(\cH)\quad\mbox{and}\quad\Omega_G^X\cap\ran X^{1/2}=\{0\}\\
\iff X\in\bB^+_0(\cH)\quad\mbox{and}\quad X^{1/2}\Omega_G^X\cap\ran X=\{0\}\\
\iff X\in\bB^+_0(\cH)\quad\mbox{and}\quad\ran\left([G]X\right))^{1/2}\cap\ran X=\{0\}\\
\iff X\in\bB^+_0(\cH)\quad\mbox{and}\quad\left([G]X\right):X^2=0.
\end{multline*}
Further we use the equality $\tau_G(X)=\tau_G(G+X)$, see statement (13) of Theorem \ref{propert}.
\end{proof}
\section{The mappings $\{\mu_G^{[n]}\},$ $\tau_G,$ and intersections of domains of unbounded self-adjoint operators}
\label{applll}
Let $A$ be an unbounded self-adjoint operator in an infinite dimensional Hilbert space $\cH$. J.von Neumann \cite[Satz 18]{Neumann} established that if $\cH$ is \textit{separable}, then there is a self-adjoint operator unitary equivalent to $A$ such that its domain has trivial intersection with the domain of $A$. Another proof of this result was proposed by J.~Dixmier in \cite{Dix}, see also \cite[Theorem 3.6]{FW}. In the case of \textit{nonseparable} Hilbert space in \cite{ES} it is constructed an example of unbounded self-adjoint operator $A$ such that for any unitary $U$ one has $\dom (U^*AU)\cap\dom A\ne \{0\}$. So, in general, the von Neumann theorem does not hold. It is established in \cite[Theorem 4.6]{ES}, that the following are equivalent for a dense operator range $\cR$ (the image of a bounded nonnegative self-adjoint operator in $\cH$ \cite{FW}) in an infinite-dimensional Hilbert space:
\begin{enumerate}
\def\labelenumi{\rm (\roman{enumi})}
\item
there is a unitary operator $U$ such that $U\cR\cap\cR=\{0\}$;
\item for every subspace (closed linear manifold) $\cK\subset \cR$ one has $\dim \cK\le\dim \cK^\perp$.
\end{enumerate}

In the theorem below we suggest another several statements equivalent to the von Neumann's theorem.
\begin{theorem}
\label{ytcgjl}
Let $\cH$ be an infinite-dimensional complex Hilbert space and let $A$ be an unbounded self-adjoint operator in $\cH$.
Then the following assertions are equivalent
\begin{enumerate}
\item there exists a unitary operator $U$ in $\cH$ such that
$$\dom(U^*AU)\cap\dom A=\{0\};$$
\item there exists an unbounded self-adjoint operator $S$ in $\cH$ such that
$$\dom S\cap\dom A=\{0\};$$
\item there exists a fundamental symmetry $J$ in $\cH$ ($J=J^*=J^{-1}$) such that
$$\dom(JAJ)\cap\dom A=\{0\};$$
\item there exists a subspace $\sM$ in $\cH$ such that
$$\sM\cap\dom A=\sM^\perp\cap\dom A=\{0\};$$
\item there exists a positive definite self-adjoint operator $B$ in $\cH$ such that
$$\dom B\supset\dom A\quad\mbox{and}\quad\clos\left\{B\dom A\right\}\cap \dom B=\{0\},$$
\item there exists a closed densely defined restriction $A_0$ of $A$ such that $\dom (AA_0)=\{0\}$ (this yields, in particular, $\dom A^2_0=\{0\}$).
 \end{enumerate}
\end{theorem}
\begin{proof}
Let $|A|=\sqrt{A^2}$. Set $G=\left(|A|+I\right)^{-2}$. Then $G\in\bB^+_0(\cH)$ and $\ran G^{1/2}=\dom A.$

According  to \cite[Proposition 3.1.]{Arl_ZAg_IEOT_2015} the following assertion for the operator range $\cR$ are equivalent
\begin{enumerate}
\def\labelenumi{\rm (\roman{enumi})}
\item There exists in $\cH$ an orthogonal projection $P$ such that
\[
\ran P\cap\cR=\{0\} \ \ \ {\rm{and}}  \ \ \ \ran (I-P)\cap\cR=\{0\} \ .
\]
\item There exists in $\cH$ a fundamental symmetry $J$  such that
\[
J\cR\cap\cR=\{0\} \ .
\]
\end{enumerate}
Now we will prove that
(2)$\Longrightarrow$(1), (3), (4), (5). The existence of self-adjoint $S$ with the property $\dom S\cap\dom A=\{0\}$ implies the existence of
$F\in\bB^+_0(\cH)$ such that $\ran F^{1/2}\cap\ran G^{1/2}=\{0\}$ (for example, take $F=(|S|+I)^{-2}$). Then the equality $F:G=0$ yields, see Proposition \ref{root} that
\[
G=(G+F)^{1/2}P(G+F)^{1/2},\; F=(G+F)^{1/2}(I-P)(G+F)^{1/2},
\]
where $P$ is orthogonal projection in $\cH$. The equalities $\ker G=\ker F=\{0\}$ imply
$$\ran P\cap\ran (G+F)^{1/2}=\ran (I-P)\cap\ran (G+F)^{1/2}=\{0\}.$$
Since $\ran G^{1/2}\subset\ran (G+F)^{1/2}$, we get
$$\ran P\cap\ran G^{1/2}=\ran (I-P)\cap\ran G^{1/2}=\{0\}.$$
Let $\sM=\ran P$, then holds (4). Put $J=P-(I-P)=2P-I$. The operator $J$ is fundamental symmetry and $J\ran G^{1/2}\cap\ran G^{1/2}=\{0\}$. This gives (3).

Since $\ker F=\{0\}$ and $F=\tau_G(F)=\tau_G(G+F)$, using Theorem \ref{propert}, equalities \eqref{ghjcn1}, \eqref{formula11}, and Theorem \ref{interzero} we obtain
\[
\ran (G+F)^{1/2}\cap \cran (G+F)^{-1/2}G^{1/2}=\{0\}.
\]
Denoting $B=(G+F)^{-1/2}$, we arrive to (5).

 Let us proof (5)$\Longrightarrow$(2). Set $X=B^{-2}$. Then $\ran X^{1/2}\supset\ran G^{1/2}$ and
\[
X\in \bB^+_0(\cH),\;
\ran X^{1/2}\cap\clos\left\{X^{-1/2}\ran G^{1/2}\right\}=\{0\}.
\]
The equivalence of conditions (16a) and (16d) of Theorem \ref{propert} implies $\ker\tau_G(X)=\{0\}.$ Since the operator $Y=\tau_G(X)$ possesses the property
$\ran Y^{1/2}\cap\ran G^{1/2}=\{0\}$, we get for $S=Y^{-2}$ that $\dom S\cap\dom A=\{0\}$.

Now we are going to prove $(4)\iff (6)$.
Suppose (6) is valid, i.e., $A_0$ is closed densely defined restriction of $A$ such that $\dom (AA_0)=\{0\}$.
Let
 $$\cU=(A-iI)(A+iI)^{-1}$$
 be the Cayley transform of $A$. $\cU$ is a unitary operator and
 \[
 A=i(I+\cU)(I-\cU)^{-1},\; \dom A=\ran (I-\cU),\ran A=\ran (I+\cU).
 \]

Let $\cU_0=(A_0-iI)(A_0+iI)^{-1}$ be the Cayley transform of $A_0$. Set $\sM\stackrel{def}{=}\ran (A_0+iI)$. Then $\cU_0=\cU\uphar\sM$,
\begin{multline*}
\dom A_0=\ran (I-\cU_0)=(I-\cU)\sM,\\
 \ran A_0=\ran (I+\cU_0)=(I+\cU)\sM.
\end{multline*}
Because $\dom A_0$ is dense in $\cH$, we get $\sM^\perp\cap\dom A=\{0\}$.
The equality $\dom (AA_0)=\{0\}$ is equivalent to
\[
\left\{\begin{array}{l}\ran A_0\cap\dom A=\{0\},\\
\ker A_0=\{0\}\end{array}
\right..
\]
The latter two equalities are equivalent to $\sM\cap\dom A=\{0\}$. Thus (4) holds. If (4) holds, then define
the symmetric restriction $A_0$ as follows
\[
 \dom A_0=(I-\cU)\sM,\; A_0=A\uphar\dom A_0
 \]
 we get $\dom(AA_0)=\{0\}$.
  The proof is complete.

\end{proof}
Let us make a few remarks.
\begin{remark}
\label{ljd1}
If (5) is true, then
\begin{enumerate}
\item
the more simple proof of the implication (5)$\Rightarrow$(2) is the observation that (5) implies $\dom B^2\cap \dom A=\{0\}$;
\item taking into account that $B^{-1}$ is bounded and $\dom A$ is dense in $\cH$, we get
 \[
 \left(\cH\ominus\clos\left\{B\dom A\right\}\right)\cap \dom B=\{0\},
 \]
if we set $\sM\stackrel{def}{=}\clos\left\{B\dom A\right\}$, we see that the inclusion $\dom A\subset\dom B$ implies (4), i.e., this is one more way to prove (5)$\Longrightarrow$(4) and (5)$\Longrightarrow$(6);
\item using the proof of Theorem \ref{ytcgjl} and equalities \eqref{ghjcn2} and \eqref{formula111}, we see that the operator $S\stackrel{def}{=}\left(B^{-1}P_{\sM^\perp}B^{-1}\right)^{-1}$ is well defined, self-adjoint positive definite, and $\dom S\cap\dom A=\{0\}$;
\item denoting $B_0=B\uphar\dom A$ and taking the closure of $B_0,$ we get the closed densely defined positive definite symmetric operator $\bar{B}_0$ (a closed restriction of $B$) such that
\[
\dom (B\bar{B}_0)=\{0\}.
\]
\end{enumerate}

\end{remark}
\begin{remark}

In the case of an infinite-dimensional Hilbert space be separable.
K.~Schm\"{u}dgen in \cite[Theorem 5.1]{schmud} established the validity of assertion (4) for an arbitrary $A$.  In \cite{Arl_ZAg_IEOT_2015} using parallel addition of operators it is shown that validity (2) for an arbitrary unbounded self-adjoint $A$ implies (4).

The first construction of a densely defined closed symmetric operator $T$ such that $\dom T^2=\{0\}$ was given by M.A.~Naimark \cite{Naimark1}, \cite{Naimark2}.
In \cite{Chern} P.~Chernoff gave an example of semi-bounded from bellow symmetric $T$ whose square has trivial domain. K.~Schm\"{u}dgen in \cite[Theorem 5.2]{schmud} proved that each unbounded self-adjoint operator $H$ has two closed densely defined restrictions $H_1$ and $H_2$ such that
\[
\dom H_1\cap\dom H_2=\{0\}\quad\mbox{and}\quad\dom H^2_1=\dom H^2_2=\{0\}.
\]
 In \cite{ArlKov_2013} the abstract approach to the construction of examples of nonnegative self-adjoint operators $\cL$ and their closed densely defined restrictions $\cL_0$
such that $\dom(\cL\cL_0)=\{0\}$ has been proposed. In \cite[Theorem 3.33]{Arl_ZAg_IEOT_2015} it is established that each unbounded self-adjoint $A$ has two closed densely defined restrictions $A_1$ and $A_2$ possessing properties
\begin{multline*}
\dom A_1\dot+\dom A_2=\dom A,\; \dom (AA_1)=\dom (AA_2)=\{0\},
\\
\dom A_1\cap\dom A^2=\dom A_2\cap\dom A^2=\{0\}.
\end{multline*}

M.~Sauter in the e-mail communication with the author suggested
another proof of the equivalence of (1) and (2) in Theorem \ref{ytcgjl}. His proof is essentially relied  on the methods developed in the paper \cite{ES}.
\end{remark}
We conclude this paper by the theorem related to the assertions (2) and (5) of Theorem \ref{ytcgjl}. The proof is based on the properties of the mappings $\{\mu^{[n]}_G\}$ and $\tau_G$.

\begin{theorem}
\label{bynth}
Let $\cH$ be an infinite dimensional separable Hilbert space and let $A$ be unbounded self-adjoint operator in $\cH$.
Then for each positive definite self-adjoint operator $S$ such that $\dom S\cap \dom A=\{0\}$ there exists a sequence $\{S_n\}$ of positive definite operators possessing properties
\begin{itemize}
\item $\dom S_n=\dom S\dot+\dom A$ $\forall n$,
\item $\clos\left\{S_n\dom A\right\}\cap \dom S_n=\{0\}$ $\forall n$,
\item $\dom S^2_n\cap\dom A=\{0\}$ $\forall n,$
\item if $\sL_n=\cH\ominus \clos\left\{S_n\dom A\right\}$, then $S=\left(S^{-1}_nP_{\sL_n}S^{-1}_n\right)^{-1}$ $\forall n$,
\item for each $f\in \dom S\dot+\dom A$ the sequence $\{||S_nf||\}_{n=1}^\infty$ is nondecreasing,
\item $\dom S=\left\{f: \sup\limits_{n\ge 1}||S_n f||<\infty\right\},$
\item  ${\rm s-R}-\lim\limits_{n\to\infty}S_n=S$, where ${\rm s-R}$ is the strong resolvent limit of operators \cite[Chapter 8, \S 1]{Ka}.

\end{itemize}
\end{theorem}
\begin{proof}
Let $G\stackrel{def}{=}(|A|+I)^{-2}$, $F\stackrel{def}{=}S^{-2}$. Then $\ran G^{1/2}=\dom A,$ $\ran F^{1/2}=\dom S$. According to Theorem \ref{propert} the equalities
\[
F=\tau_G(F)=\tau_G(G+F)
\]
are valid.
Set
\[
F_n=\mu_G^{[n]}(G+F), n=0,1,\ldots.
\]
Then $\{F_n\}$ is non-increasing sequence of operators, $\tau_G(F_n)=F$, and
$$s-\lim\limits_{n\to\infty}F_n=F.$$
 Due to the L\"{o}wner-Heinz inequality we have that the sequence of operators $\{F^{1/2}_n\}_{n=1}^\infty$ is non-increasing. In addition
 \[
 s-\lim\limits_{n\to\infty}F^{1/2}_n=F^{1/2}.
 \]
Since $\ran F^{1/2}_0=\ran G^{1/2}\dot+\ran F^{1/2}$, Proposition \ref{polez} yields $\ran F^{1/2}_n=\ran G^{1/2}\dot+\ran F^{1/2}$ for all natural numbers $n$.
Now define
\[
S_n=F^{-1/2}_n,\; n=0,1,\ldots.
\]
Then  for all $n$:
$$\dom S_n=\ran F^{1/2}_n=\ran G^{1/2}\dot+\ran F^{1/2}=\dom A\dot+\dom S,$$
the sequences of unbounded nonnegative self-adjoint operators $\{S^2_n\}$ and $\{S_n\}$ are non-decreasing,
\[
\lim\limits_{n\to\infty} S^{-1}_n=S^{-1},\;\lim\limits_{n\to\infty} S^{-2}_n=S^{-2}.
\]
The latter means, that
\[
{\rm s-R}-\lim\limits_{n\to\infty}S_n=S,\; {\rm s-R}-\lim\limits_{n\to\infty}S^2_n=S^2.
\]
Taking into account that $\tau_G(F_n)=F$ and using statement 2) of Proposition \ref{polez} we conclude that the equality
\[
\cran F^{1/2}_n \cap \clos\{F^{-1/2}_n\ran G^{1/2}\}=\{0\}
\]
holds for each $n\in\dN$.
Hence $\clos\left\{S_n\dom A\right\}\cap \dom S_n=\{0\}$ and $\dom S^2_n\cap\dom A=\{0\}$ for all natural numbers $n.$
Set
\[
\sL_n:=\cH\ominus \clos\left\{S_n\dom A\right\}
=\cH\ominus\clos\{F^{-1/2}_n\ran G^{1/2}\}.
\]
Taking in mind the equality (see \eqref{ghjcn2} and \eqref{formula111})
\[
F=\tau_G(F_n)=F^{1/2}_nP_{\sL_n}F^{1/2}_n,
\]
we get
$S=\left(S^{-1}_nP_{\sL_n}S^{-1}_n\right)^{-1}$ for all $n\in\dN$.

Let $f\in\dom S=\ran F^{1/2}$. Since $F_n\ge F$ for all $n\in\dN$, we have $F^{-1}_n\le F^{-1}$, i.e., $||S_n f||\le ||S f||$ for all $n$.

Suppose that $||S_n f||\le C$ for all $n$. Then there exists a subsequence of vectors $\{S_{n_k}f\}_{k=1}^\infty$ that
converges weakly to some vector $\varphi$ in $\cH$, i.e,
\[
\lim\limits_{k\to\infty} (S_{n_k} f, h)=(\varphi, h) \quad\mbox{for all} \quad h\in\cH.
\]
Further for all $g\in \cH$
\begin{multline*}
(f, g)=(F^{1/2}_{n_k}S_{n_k}f,g)=(S_{n_k}f,F^{1/2}_{n_k}g)\\
=(S_{n_k}f,F^{1/2}g)+ (S_{n_k}f,F^{1/2}_{n_k}g-F^{1/2}g)\rightarrow (\varphi,F^{1/2}g)=(F^{1/2}\varphi, g).
\end{multline*}
It follows that $f\in\dom S$.

Thus, we arrive to the equality
$\dom S=\left\{f: \sup\limits_{n\ge 1}||S_n f||<\infty\right\}.$
The proof is complete.
\end{proof}

\end{document}